\documentclass[12pt]{article}
\usepackage[utf8]{inputenc}
\usepackage[T1]{fontenc}
\usepackage[french, english]{babel}
\usepackage{amsmath,amssymb}
\usepackage{amsthm}
\usepackage{geometry}
\usepackage{hyperref}
\usepackage{authblk}
\usepackage{titlesec}

\newtheorem{theorem}{Theorem}[section]

\newtheorem{lemma}[theorem]{Lemma}

\newtheorem{definition}{Definition}[section]
\newtheorem{proposition}{Proposition}[section]
\newtheorem{remark}{Remark}[section]
\newtheorem{remark definition}{Remark Definition}[section]
\titleformat{\section}{\large\bfseries}{\thesection.}{1em}{}

\title{Generic three-parameter unfolding of a completely degenerate germ of a vector field on $\mathbb{R}^3$}
\author[Samuel Diangitukulu Ndimba$^{1,}$, Aimable Sayinzoga$^{3}$, Bertuel Tangue Ndawa$^{4,5}$, Jonathan Ibia Kobongo$^{2,3}$, Joseph Mpulu Kikwasini$^{2,3}$]{
\textbf{Samuel Diangitukulu Ndimba$^{1,}$, Aimable Sayinzoga$^{3}$, Bertuel Tangue Ndawa$^{4,5}$,}\\\textbf{Jonathan Ibia Kobongo$^{2,3}$, Joseph Mpulu Kikwasini$^{2,3}$}\\
$^{1}$Department of Petroleum Resources Management and Economics, Faculty of Petroleum, Gas and Renewable Energies,\\ University of Kinshasa, Democratic Republic of the Congo,\\ ORCID iD: https://orcid.org/0000-0001-6755-8859\\
$^{2}$Center for Research on Mathematics Education in the Democratic Republic of the Congo,\\ORCID iD: https://orcid.org/0009-0009-2810-1739\\ ORCID iD: https://orcid.org/0009-0004-2147-6320\\
$^{3}$Department of Mathematics, Statistics, and Computer Science, Faculty of Science and Technology, \\University of Kinshasa, Democratic Republic of the Congo,\\ORCID iD: https://orcid.org/0009-0009-2810-1739\\ ORCID iD: https://orcid.org/0009-0004-2147-6320\\
$^{4}$Department of Computer Engineering, University Institute of Technology, University of Ngaoundere, Cameroon,\\ ORCID iD: https://orcid.org/0000-0001-8995-9522\\
$^{5}$Institut des Hautes Études Scientifiques, Université Paris-Saclay, France,\\ ORCID iD: https://orcid.org/0000-0001-8995-9522\\
}
\begin{document}
\maketitle

\begin{center}
	\textit{This article was initiated by the late Professor Aimable Sayinzoga.
		Although the anticipated results could not be published during his lifetime,
		significant findings have since been obtained. These results are now being
		published as a tribute to his fundamental contributions, which is why his
		name remains among the authors.}
\end{center}

\vspace{1em}

\begin{abstract}
	In this work, we analyze a completely degenerate germ of a real vector field
	at the origin of $\mathbb{R}^3$. We construct a generic three-parameter
	unfolding of this germ and identify the associated local bifurcations. We
	prove that the phase portrait of the unfolded system contains a
	one-parameter family of heteroclinic orbits connecting two hyperbolic
	equilibrium points.
	
	\medskip
	
	\noindent\textbf{Keywords:} unfolding, induced unfolding, local bifurcation
	set, equivalence and conjugacy of unfoldings.
	
	\medskip
	
	\noindent\textbf{AMS Mathematics Subject Classification:} 37C10, 37G10,
	34C23, 58K05, 58K45.
\end{abstract}

\section{Introduction}

The classification of singular vector fields of codimension three in
$\mathbb{R}^n$, with $n\geq 3$, is a central topic in singularity theory.
In particular, singular germs of vector fields on $\mathbb{R}^3$ whose linear
jet is nilpotent have been extensively studied; see, for instance, \cite{i}.
Such germs arise naturally in three-parameter unfoldings of singular germs of
codimension three defined on $\mathbb{R}^n$. Moreover, singular vector fields
with nilpotent linear part in $\mathbb{R}^3$ occur in the modelling of various
physical and biological phenomena, notably reaction--diffusion systems and the
propagation of nerve impulses in neurons.

Up to an affine isomorphism, we may assume that the first-order jet of the
singular germ $X_0$ at the origin is given by
$$
j_1X_0(0)
=
x_2\frac{\partial}{\partial x_1}
+
x_3\frac{\partial}{\partial x_2}.
$$
It is therefore sufficient to study the generic unfolding, at the origin of
$\mathbb{R}^3$, of the vector field
$$
X_0
=
x_2\frac{\partial}{\partial x_1}
+
x_3\frac{\partial}{\partial x_2}
+
f_0(x_1,x_2,x_3)\frac{\partial}{\partial x_3},
$$
where $f_0\in C^l(\mathbb{R}^3)$, $l\geq 5$, satisfies
$$
j_1f_0(0)=0,
\qquad
\frac{\partial^2 f_0}{\partial x_1^2}(0)=a\neq 0.
$$
After applying an automorphism of the form $\varphi(x)=|a|x$, followed by the
symmetry $T(x)=-x$, we may assume, without loss of generality, that $a=2$.
Thus, we consider a germ $X_0$ whose second-order jet satisfies
$$
j_1f_0(0)=0,
\qquad
j_2f_0(0)=\sum_{i,j=1}^{3}a_{ij}x_ix_j,
$$
with
$$
a_{11}=1,
\qquad
a_{ij}=a_{ji},
\qquad
i,j=1,2,3.
$$

When $f_0(x)=x_1^2$, the phase portrait of the germ $(X_0,0)$ contains only
one invariant manifold $\gamma$ passing through the origin. This manifold is of
class $C^1$, but it is not of class $C^2$ along the $x_1$-axis at the origin;
see \cite{iv}. The purpose of the present work is to construct a generic
three-parameter unfolding of the completely degenerate germ $X_0$ and to
analyze the local bifurcations arising from it. In particular, we establish the
existence of a family of heteroclinic orbits, which illustrates the richness of
the phase portrait near the origin.

Before stating our results in a precise form and proceeding to their
proofs, we first recall the necessary definitions, fix the notation to be
used throughout the paper, and collect several preliminary results that
will be needed in the sequel.
	       	
	\section{Preliminaries}
	\label{Sect2_GPU}
	We recall in this section the notions and notation that will be used throughout
	the paper. For further details, we refer the reader to
	\cite{FR,GP,PL,i,iv,v}.
	
	\subsection{Definitions and notation}
	
	Throughout this paper, $k$ and $\ell$ are positive integers, and $M$ is an
	$n$-dimensional $C^m$-manifold, with $m\geq \ell$ and $n\geq 1$. We denote by
	$
	p_1:\mathbb{R}^k\times M\longrightarrow \mathbb{R}^k,
	\qquad
	p_2:\mathbb{R}^k\times M\longrightarrow M
	$
	the canonical projections. A local coordinate system around a point $x\in M$
	will be denoted by $(x_1,\dots,x_n)$. If $a\in \mathbb{R}^r$, $r\geq 1$, we
	write
	$
	a=(a_1,\dots,a_r).
	$
	Moreover, $C^m(M,S_n(\mathbb{R}))$ denotes the space of $C^m$-maps from $M$
	to the vector space $S_n(\mathbb{R})$ of real symmetric $n\times n$ matrices.
	
	\begin{definition}[Unfolding]
		Let $X_0\in \mathfrak{X}^{\ell}(M)$ be a vector field on $M$. A
		$k$-parameter $C^\ell$-unfolding of $X_0$ is a $C^\ell$-map
		$
		X:\mathbb{R}^k\times M\longrightarrow TM
		$
		such that
		$
		X(0,x)=X_0(x)$, $x\in M,$
		and, for each fixed $\lambda\in \mathbb{R}^k$, the map
		$
		x\longmapsto X(\lambda,x)
		$
		defines a vector field $X_\lambda\in \mathfrak{X}^{\ell}(M)$.
		
		We denote by
		$
		\mathfrak{X}^{\ell}(M\mid \mathbb{R}^k)
		$
		the set of all $k$-parameter $C^\ell$-unfoldings of vector fields on $M$.
	\end{definition}
	
	\begin{remark}
		The preceding definition also applies when the parameter space
		$\mathbb{R}^k$ is replaced by an open neighborhood of $0$ in
		$\mathbb{R}^k$.		
		If $X\in \mathfrak{X}^{\ell}(M\mid \mathbb{R}^k)$ and
		$\lambda\in \mathbb{R}^k$ is fixed, then $X_\lambda$ denotes the vector
		field on $M$ given by
		$
		(X_\lambda)_x=X(\lambda,x),
		\;
		x\in M.
		$
	\end{remark}
	
	\begin{definition}[$C^s$-equivalent and $C^s$-conjugate unfoldings]
		Let $X,Y\in \mathfrak{X}^{\ell}(M\mid \mathbb{R}^k)$. We say that $X$ and
		$Y$ are $C^s$-equivalent if there exist a $C^s$-diffeomorphism
		$h:\mathbb{R}^k\longrightarrow \mathbb{R}^k$
		with $h(0)=0$, and a $C^s$-map $
		\phi:\mathbb{R}^k\times M\longrightarrow M$
		such that, for each $\lambda\in \mathbb{R}^k$, the map $
		\phi_\lambda:M\longrightarrow M$,
		$\phi_\lambda(x)=\phi(\lambda,x),$
		is a $C^s$-diffeomorphism satisfying 
		$$\phi_0=Id \quad\text{and}\quad
		Y_{h(\lambda)}
		=
		(\phi_\lambda)_*X_\lambda,
		\qquad
		\lambda\in \mathbb{R}^k.
		$$
		Equivalently, if
		$$
		\Phi:\mathbb{R}^k\times M\longrightarrow \mathbb{R}^k\times M,
		\qquad
		\Phi(\lambda,x)=(h(\lambda),\phi(\lambda,x)),
		$$
		with the following properties
		$$
		\Phi(0,)=(0,Id)\quad\text{and} \quad (Y_{h(\lambda)})_{\phi(\lambda,x)}
		=
		(\phi_\lambda)_*x(X_\lambda)_x.
		$$
		
		If, moreover, $
		h(\lambda)=\lambda$, $\lambda\in \mathbb{R}^k,$
		then $X$ and $Y$ are said to be $C^s$-conjugate.
	\end{definition}
	
	\begin{remark}
		The relations of $C^s$-equivalence and $C^s$-conjugacy are equivalence
		relations on $\mathfrak{X}^{\ell}(M\mid \mathbb{R}^k)$. It can be see as an action of $C^s$-diffeomorphisms on $
		\mathfrak{X}^{\ell}(M\mid \mathbb{R}^k)
		$. A similar action, without parameters, on the set of Cherry vector fields was studied in \cite{TNB4}. This action induces, in turn, a conjugation action on
		the set of circle maps with a flat interval.
	\end{remark}
	
	\begin{definition}[Germs and jets]
		Let $X_0$ be a vector field on $M$, and let $x_0\in M$. The germ of
		$X_0$ at $x_0$ is the equivalence class of the pair $(X_0,x_0)$, where two
		vector fields are identified if they coincide on an open neighborhood of
		$x_0$. If $X_0(x_0)=0,$ then the germ $(X_0,x_0)$ is called a singular germ. 
		Similarly, if $X\in \mathfrak{X}^{\ell}(M\mid \mathbb{R}^k)	$
		and $(\lambda_0,x_0)\in \mathbb{R}^k\times M$ satisfies $X_{\lambda_0}(x_0)=0,
		$
		then one defines the singular germ of the unfolding $X$ at
		$(\lambda_0,x_0)$.
		
		Let $\ell'\leq \ell$. The $\ell'$-jet of a $C^\ell$-vector field at a
		point records the values of the component functions of the vector field and
		all their partial derivatives up to order $\ell'$ at that point. More
		precisely, two vector fields are said to have the same $\ell'$-jet at
		$x_0$ if, in local coordinates, their component functions have the same
		partial derivatives up to order $\ell'$ at $x_0$.
		
		In this work, following \cite[Def.~1.2]{v}, we shall mainly consider jets
		of vector fields on $\mathbb{R}^3$ at the origin. If $X_0$ is a vector
		field on $\mathbb{R}^3$ such that $X_0(0)=0$, we denote its $\ell'$-jet by
		$J^{\ell'}X_0=J^{3,\ell'}X_0.$
		The set of all such $\ell'$-jets is denoted by $J^{\ell'}(\mathbb{R}^3).$
		For a function $f:\mathbb{R}^3\to \mathbb{R}$, the $\ell'$-jet of $f$ at
		the origin is denoted by $J^{\ell'}f=J^{3,\ell'}f.$		
		Formally, the $\ell'$-jet of a function, respectively of a vector field,
		is represented by the truncation at order $\ell'$ of its Taylor expansion,
		respectively of the Taylor expansions of its component functions.
		
		These notions extend naturally to unfoldings. If $
		X\in \mathfrak{X}^{\ell}(\mathbb{R}^3\mid \mathbb{R}^k),$
		then the $\ell'$-jet unfolding $\widetilde{J}^{\ell'}X$ is defined by
		$(\widetilde{J}^{\ell'}X)_\lambda=J^{\ell'}X_\lambda.$
		Thus, an unfolding $Y$ induces a map
		$$
		\widetilde{J}^{\ell'}Y:
		\mathbb{R}^k\longrightarrow J^{\ell'}(\mathbb{R}^3),
		\qquad
		\lambda\longmapsto J^{\ell'}Y_\lambda.
		$$
		This map allows one to express the genericity of the family $Y$ in terms of
		transversality to a suitably chosen submanifold of
		$J^{\ell'}(\mathbb{R}^3)$; see \cite{iv}.
	\end{definition}
	
	\begin{definition}[$C^s$-equivalent and $C^s$-conjugate singular germs of unfoldings]
		Let $u=(\lambda_0,x_0)$, $v=(\mu_0,y_0)$ be points of $\mathbb{R}^k\times M$, and let $X,Y\in \mathfrak{X}^{\ell}(M\mid \mathbb{R}^k)$ be such that $X_{\lambda_0}(x_0)=0$, $Y_{\mu_0}(y_0)=0.$
		The singular germs $(X,u)$ and $(Y,v)$ are said to be $C^s$-equivalent
		if there exist open neighborhoods $U$ of $u$ and $V$ of $v$, together with
		a $C^s$-diffeomorphism $\Phi:U\longrightarrow V$, 		$\Phi(\lambda,x)=(h(\lambda),\phi(\lambda,x))$,
		such that
		\begin{enumerate}
			\item $\Phi(u)=v$;
			\item $\Phi$ induces a $C^s$-equivalence between the restrictions of
			$X$ and $Y$ to $U$ and $V$, respectively.
		\end{enumerate}
		If, in addition, $h(\lambda)=\lambda$, then the germs $(X,u)$ and $(Y,v)$
		are said to be $C^s$-conjugate.
	\end{definition}
	
\begin{definition}[Induced unfolding]
Given two unfoldings $X$ and $Y$, we say that the $C^l$-unfolding $Y$ is induced by $X$ if there exists a $C^l$-smooth map $h: \mathbb{R}^k \to \mathbb{R}^k$ such that $Y(x, \lambda) = X(h(\lambda),x)$ for all $(\lambda,x) \in \mathbb{R}^k \times M$. And, the germ $(X,u)$ induces $(Y,v)$ if there exists an open neighborhood $U$ of $u$, an open neighborhood $V$ of $v$, and a $C^s$-map $h : p_1(V) \to h\left(p_1(V)\right)\subset p_1(U)$ such that  
\begin{enumerate}
	\item $h \circ p_1 (v) = p_1(u)$;
	\item $Y(\lambda, x) = X(h(\lambda), x)$ for all $(\lambda, x)\in \mathbb{R}^k \times M$.
\end{enumerate}
An unfolding of a germ $(X_0, x_0)$ is a germ $(Y, (0, x_0))$ of an unfolding $Y$ of a representative of $(X_0, x_0)$, see \cite{iv}.
\end{definition}

\begin{definition}[Bifurcation of an unfolding]
	Let $X\in \mathfrak{X}^l (M\vert \mathbb{R}^{k})$.
\begin{enumerate}
	\item[-] A vector $\lambda\in \mathbb{R}^k$ is a bifurcation value of $X$ when $X_\lambda$ is not structurally stable. The bifurcation set of $X$ is the set of bifurcation values of $X$.
	\item[-] An ordered pair  $(\lambda_0 , x_{0}) \in \mathbb{R}^k \times M$ is a bifurcation point if the singular germ $(X_{\lambda_0} , x_{0} )$ is structurally unstable.
	\item[-] The bifurcation diagram of $X$ consists of the bifurcation points, it is composed of invariant sets of $X$ where there is a lack of structural stability.
\end{enumerate}
\end{definition}
		
	\section{Generic unfoldings of singular germs of vector fields on $\mathbb{R}^3$ of codimension 3}\label{Sect3_GPU}
	\subsection{Three-parameter generic unfoldings}
We fix an integer $k\geq 3$,  the maps $f_0\in C^m(\mathbb{R}^3)$, $f, f_1,f_2,f_3\in C^{\infty}( \mathbb{R}^k\times\mathbb{R}^3)$ and an unfolding $X\in \mathfrak{X}^m (\mathbb{R}^3\vert \mathbb{R}^k )$ of a vector field $X_0 \in \mathfrak{X}^l (\mathbb{R}^3)$ such that  for every $x=(x_1 , x_2 , x_3 )\in \mathbb{R}^3$ and $\lambda\in \mathbb{R}^k$ the following holds: 
\begin{equation}\label{GSUH1}\tag{H1}
\begin{array}{l}
X_0=x_2\frac{\partial}{\partial x_1}+x_3\frac{\partial}{\partial x_2} + f_0 (x)\frac{\partial}{\partial x_3},
 \\f_0(x) = \sum_{i,j=1}^{3} x_i x_j a_{ij}(x),\; (a_{ij})\in C^m(M, S_n(\mathbb{R})),\mbox{ and } a_{11}(0)=1;
\end{array}
\end{equation}

\begin{equation}\label{GSUH2}\tag{H2}
\begin{array}{l}
X_\lambda (x)= x_2 \frac{\partial}{\partial x_1}+x_3 \frac{\partial}{\partial x_2} + f(\lambda,x)\frac{\partial}{\partial x_3}+\sum_{i=1}^{3} f_i(\lambda,x)\frac{\partial}{\partial x_i}, \\f_i (0, \cdot )=0,\quad f (0, \cdot )=f_0(\cdot);
\end{array}
\end{equation}		
\begin{equation}\label{GSUH3}\tag{H3}
\begin{array}{l}f (\lambda , x)+ f_3 (\lambda , x) = f_{30} (\lambda)+x_1 f_{31}(\lambda ) + x_{1}^{2}f_{32}(\lambda ,x_1)+x_2 h_1 (\lambda , x)\\
\hspace{3.6cm}+ x_3 h_2 (\lambda ,x ),\\ 
f_{3i},\; h_1\; h_2\in C^m\left(\mathbb{R}^k\times\mathbb{R}^3\right),\;  i=1,2,3,\;f_{32} (0,0)=1,\\  f_{30}(0,0) = f_{3 1}(0,0) =h_i (0,0),\; i=1,2\end{array}
\end{equation}

\begin{proposition}
	Let $X\in \mathfrak{X}^{\ell}(\mathbb{R}^3\mid \mathbb{R}^k)$
	be an unfolding satisfying \eqref{GSUH2}. Then the germ
	$(X,(0,0))$ is $C^{\ell-1}$-conjugate to the germ $(Y,(0,0))$ of an
	unfolding $Y$ of $X_0$, where $Y$ admits the local expression
	$$
	Y_\lambda
	=
	y_2\frac{\partial}{\partial y_1}
	+
	y_3\frac{\partial}{\partial y_2}
	+
	g(\lambda,y)\frac{\partial}{\partial y_3}.
	$$
	Moreover, $Y$ satisfies \eqref{GSUH2} and is of class $C^{\ell-2}$.
\end{proposition}

\begin{proof}
	We write
	$$
	X_\lambda
	=
	(x_2+f_1(\lambda,x))\frac{\partial}{\partial x_1}
	+
	(x_3+f_2(\lambda,x))\frac{\partial}{\partial x_2}
	+
	(f(\lambda,x)+f_3(\lambda,x))\frac{\partial}{\partial x_3}.
	$$
	
	We look for a local change of coordinates of the form
	$$
	\Phi:\mathbb{R}^k\times \mathbb{R}^3
	\longrightarrow
	\mathbb{R}^k\times \mathbb{R}^3,
	\qquad
	\Phi(\lambda,x)=(\lambda,y),
	$$
	where
	$$
	y_1=x_1,\qquad
	y_2=x_2+f_1(\lambda,x),\qquad
	y_3=x_3+r(\lambda,x).
	$$
	
	Since $f_1(0,\cdot)=0$ and we choose $r$ such that $r(0,\cdot)=0$, the
	differential of $\Phi$ at $(0,0)$ is the identity. Therefore, by the inverse
	function theorem, $\Phi$ is a local diffeomorphism at the origin.
	
	We now choose $r$ so that the second component of $\Phi_*X$ is exactly
	$y_3$. Since
	$$
	X(y_2)
	=
	X(x_2+f_1),
	$$
	we impose
	$$
	X(x_2+f_1)=x_3+r.
	$$
	Hence
	$$
	r
	=
	f_2
	+
	\frac{\partial f_1}{\partial x_1}(x_2+f_1)
	+
	\frac{\partial f_1}{\partial x_2}(x_3+f_2)
	+
	\frac{\partial f_1}{\partial x_3}(f+f_3).
	$$
	Because $f_1,f_2,f,f_3$ are of class $C^\ell$, the function $r$ is of class
	$C^{\ell-1}$. Moreover, using $f_1(0,\cdot)=f_2(0,\cdot)=f_3(0,\cdot)=0$,
	we have $r(0,\cdot)=0.$
	
	Let $Y=\Phi_*X.$ 	Then, by construction,	$Y(y_1)=y_2$, 	$Y(y_2)=y_3.$
	The third component is given by $Y(y_3)=g(\lambda,y),$ 	where $g$ is defined locally by $g\circ \Phi	=X(x_3+r).$
	Equivalently,
	$$
	g\circ \Phi
	=
	f+f_3
	+
	\frac{\partial r}{\partial x_1}(x_2+f_1)
	+
	\frac{\partial r}{\partial x_2}(x_3+f_2)
	+
	\frac{\partial r}{\partial x_3}(f+f_3).
	$$
	
	Thus, in the coordinates $(\lambda,y)$, the unfolding $Y$ has the form
	$$
	Y_\lambda
	=
	y_2\frac{\partial}{\partial y_1}
	+
	y_3\frac{\partial}{\partial y_2}
	+
	g(\lambda,y)\frac{\partial}{\partial y_3}.
	$$
	
	Since $\Phi(0,x)=(0,x)$, it follows that $Y_0=X_0.$
	
	Therefore $Y$ is an unfolding of $X_0$. Finally, because $\Phi$ is of class
	$C^{\ell-1}$ and the push-forward involves one derivative of $\Phi$, the
	vector field $Y$ is of class $C^{\ell-2}$.
\end{proof}

\begin{remark}
	Using the preceding reduction and the implicit function theorem, under
	assumptions \eqref{GSUH1}, \eqref{GSUH2}, and \eqref{GSUH3}, one may work,
	up to local conjugacy and a local change of parameters, in the normalized
	case $f_1=f_2=f_{31}=0.$
\end{remark}
\begin{proposition}
	Let $X\in \mathfrak{X}^{\ell}(\mathbb{R}^3\mid \mathbb{R}^3)$
	satisfy \eqref{GSUH1}--\eqref{GSUH3}. Assume moreover that
	$f_1=f_2=f_{31}=0.$
	If the $2$-jet extension of $X$ is transverse to the corresponding
	codimension-three stratum of $J^2X_0$, then the germ of $X$ at the origin of
	$\mathbb{R}^3\times \mathbb{R}^3$ is $C^\ell$-induced by the germ at the
	origin of an unfolding $Y$ of $X_0$ having the local expression
	$$
	Y_\mu
	=
	x_2\frac{\partial}{\partial x_1}
	+
	x_3\frac{\partial}{\partial x_2}
	+
	\left(
	\mu_1+\mu_2x_2+\mu_3x_3+x_1^2+G(\mu,x)
	\right)
	\frac{\partial}{\partial x_3},
	$$
	where
	$$
	J^1G(0,0)=0,
	\qquad
	\frac{\partial^2G}{\partial x_1^2}(0,0)=0.
	$$
\end{proposition}
\begin{proof}
	Under the normalization
	$$
	f_1=f_2=f_{31}=0,
	$$
	hypothesis \eqref{GSUH3} gives
	$$
	f(\lambda,x)+f_3(\lambda,x)
	=
	f_{30}(\lambda)
	+
	x_1^2f_{32}(\lambda,x_1)
	+
	x_2h_1(\lambda,x)
	+
	x_3h_2(\lambda,x).
	$$
	
	The transversality assumption on the $2$-jet of $X$ implies that
	$$
	\det
	\left(
	\frac{\partial(f_{30},h_1(\cdot,0),h_2(\cdot,0))}
	{\partial(\lambda_1,\lambda_2,\lambda_3)}(0)
	\right)
	\neq 0.
	$$
	Therefore, the map
	$$
	\varphi:\mathbb{R}^3\longrightarrow \mathbb{R}^3,
	\qquad
	\lambda\longmapsto \mu
	=
	\bigl(
	f_{30}(\lambda),
	h_1(\lambda,0),
	h_2(\lambda,0)
	\bigr)
	$$
	is a local diffeomorphism at the origin.
	
	We define an unfolding $Y$ with parameter $\mu$ by
	$$
	Y_\mu
	=
	x_2\frac{\partial}{\partial x_1}
	+
	x_3\frac{\partial}{\partial x_2}
	+
	\left(
	\mu_1+\mu_2x_2+\mu_3x_3+x_1^2+G(\mu,x)
	\right)
	\frac{\partial}{\partial x_3},
	$$
	where
	$$
	G(\mu,x)
	=
	x_1^2
	\left(
	f_{32}(\varphi^{-1}(\mu),x_1)-1
	\right)
	+
	x_2\overline{h}_1(\varphi^{-1}(\mu),x)
	+
	x_3\overline{h}_2(\varphi^{-1}(\mu),x),
	$$
	with
	$$
	\overline{h}_i(\lambda,x)
	=
	h_i(\lambda,x)-h_i(\lambda,0),
	\qquad i=1,2.
	$$
	
	Then, for $\mu=\varphi(\lambda)$, we have $Y_{\varphi(\lambda)}=	X_\lambda.$
	Equivalently, $
	X(\lambda,x)=Y(\varphi(\lambda),x).$
	Hence the germ of $X$ is induced by the germ of $Y$ via the local
	diffeomorphism $\varphi$.
	
	It remains to verify the conditions on $G$. Since
	$$
	f_{32}(0,0)=1,
	\qquad
	\overline{h}_1(0,0)=\overline{h}_2(0,0)=0,
	$$
	we obtain $J^1G(0,0)=0.$ Moreover,
	$$
	\frac{\partial^2}{\partial x_1^2}
	\left[
	x_1^2
	\left(
	f_{32}(\varphi^{-1}(\mu),x_1)-1
	\right)
	\right]_{(0,0)}
	=
	2\bigl(f_{32}(0,0)-1\bigr)
	=
	0.
	$$
	Therefore,
	$$
	\frac{\partial^2G}{\partial x_1^2}(0,0)=0.
	$$
	
	Finally, for $\mu=0$ we have $\lambda=\varphi^{-1}(0)=0$, and so $Y_0=X_0.$
	Thus $Y$ is an unfolding of $X_0$.
\end{proof}

Thus, using the normal form of the $3$-jet of $X$ at the origin obtained
from the adjoint representation
$$
H: J^s(\mathbb{R}^3\times \mathbb{R}^3)
\longrightarrow
J^s(\mathbb{R}^3\times \mathbb{R}^3),
\qquad
U\longmapsto J^s([U,Z])(0),
$$
where
$$
Z
=
x_2\frac{\partial}{\partial x_1}
+
x_3\frac{\partial}{\partial x_2}
+
x_1^2\frac{\partial}{\partial x_3},
$$
we obtain the following normal form.

\begin{proposition}
	Let $X\in \mathfrak{X}^{\ell}(\mathbb{R}^3\mid \mathbb{R}^3)$
	be a generic unfolding of $X_0$ satisfying hypotheses
	\eqref{GSUH1}--\eqref{GSUH3}. Then the germ
	$(X,(0,0))$ at the origin of
	$\mathbb{R}^3\times \mathbb{R}^3$ is locally conjugate, up to a local
	change of parameters, to the germ at the origin of an unfolding $Y$ of
	$X_0$ with local expression
	\begin{equation}
		\label{GSUH4}
		\tag{H4}
		Y
		=
		x_2\frac{\partial}{\partial x_1}
		+
		x_3\frac{\partial}{\partial x_2}
		+
		\left(
		\mu_1+\mu_2x_2+\mu_3x_3+x_1^2+g(\mu,x)
		\right)
		\frac{\partial}{\partial x_3},
	\end{equation}
	where the $3$-jet of $g$ at the origin is given by
	\begin{align*}
		J^3g(0,0)
		=
		x_1^2
		\left(
		a_{10}(\mu)+\sum_{i=1}^3 x_i a_{1i}
		\right)
		+
		x_1x_2
		\left(
		b+a_{20}(\mu)
		\right) 
		+
		x_1x_3
		\left(
		c+a_{30}(\mu)+\sum_{i=1}^3 x_i a_{3i}
		\right).
	\end{align*}
	Here
	$$
	b
	=
	2\frac{\partial^2 f_0}{\partial x_1\partial x_2}(0),
	\qquad
	c
	=
	2\frac{\partial^2 f_0}{\partial x_1\partial x_3}(0),
	$$
	with
	$|b|=1$, $c\neq 0.$
	Moreover, for $r=1,2,3$, the functions $a_{r0}$ are linear in the
	parameters:
	$$
	a_{r0}(\mu)
	=
	\sum_{i=1}^3 \mu_i a_{r0i},
	\qquad
	a_{r0i}\in \mathbb{R}.
	$$
\end{proposition}

\begin{proof}
	By the previous reduction, after a local conjugacy and a local change of
	parameters, the unfolding $X$ may be written in the form
	$$
	X
	=
	x_2\frac{\partial}{\partial x_1}
	+
	x_3\frac{\partial}{\partial x_2}
	+
	\left(
	\mu_1+\mu_2x_2+\mu_3x_3+x_1^2+g(\mu,x)
	\right)
	\frac{\partial}{\partial x_3},
	$$
	with
	$$
	J^1g(0,0)=0,
	\qquad
	\frac{\partial^2g}{\partial x_1^2}(0,0)=0.
	$$
	
	The remaining freedom in the choice of local coordinates is described, at
	the jet level, by the adjoint action $H(U)=J^3([U,Z])(0),$
	where 
	$$
	Z=x_2\frac{\partial}{\partial x_1}
	+x_3\frac{\partial}{\partial x_2}
	+x_1^2\frac{\partial}{\partial x_3}.
	$$
	This adjoint representation determines the tangent space to the orbit of
	the normal form under local changes of coordinates.
	
	Computing the quotient of the space of $3$-jets by the image of $H$, one
	obtains a complementary space generated by the monomials
	$$
	x_1^2,\quad
	x_1^2x_i,\quad
	x_1x_2,\quad
	x_1x_3,\quad
	x_1x_3x_i,
	\qquad i=1,2,3.
	$$
	Consequently, the $3$-jet of the perturbation $g$ can be reduced to the
	form
	\begin{align*}
		J^3g(0,0)
		=
		x_1^2
		\left(
		a_{10}(\mu)+\sum_{i=1}^3 x_i a_{1i}
		\right)
		+
		x_1x_2
		\left(
		b+a_{20}(\mu)
		\right)
		+
		x_1x_3
		\left(
		c+a_{30}(\mu)+\sum_{i=1}^3 x_i a_{3i}
		\right).
	\end{align*}
	
	The constants $b$ and $c$ are determined by the second-order part of the
	unperturbed vector field $X_0$. More precisely,
	$$
	b
	=
	2\frac{\partial^2 f_0}{\partial x_1\partial x_2}(0),
	\qquad
	c
	=
	2\frac{\partial^2 f_0}{\partial x_1\partial x_3}(0).
	$$
	By the genericity assumptions, one has $|b|=1$, $c\neq 0.$
	
	Finally, since the unfolding is three-parameter and generic, the
	parameter-dependent coefficients appearing in the $3$-jet may be written,
	up to higher-order terms in the parameters, as linear functions of
	$\mu=(\mu_1,\mu_2,\mu_3)$:
	$$
	a_{r0}(\mu)=\sum_{i=1}^3 \mu_i a_{r0i},
	\qquad r=1,2,3.
	$$
	This gives the stated normal form.
\end{proof}
\subsection{Local Bifurcations of a Generic Three-Parameter Unfolding}

In this subsection, we study unfoldings $X\in \mathfrak{X}^{m}(\mathbb{R}^{3}\mid \mathbb{R}^{3})$
having the local expression
\begin{equation}
	\label{GSUH5}
	\tag{H5}
	\begin{array}{l}
		X(\lambda,x)=X_{\lambda}(x)
		=
		x_2\frac{\partial}{\partial x_1}
		+
		x_3\frac{\partial}{\partial x_2}
		+
		f_{\lambda}(x)\frac{\partial}{\partial x_3},
		\\[0.2cm]
		\begin{cases}
			f_{\lambda}(x)
			=
			\lambda_1+\lambda_2x_2+\lambda_3x_3+x_1^2
			+
			\displaystyle\sum_{i,j=1}^{3}x_ix_jA_{ij}(\lambda,x),
			\\[0.2cm]
			A_{ij}=A_{ji},
			\qquad
			A_{ij}\in C^m(\mathbb{R}^3\times \mathbb{R}^3),
			\\[0.2cm]
			A_{11}(0,0)=A_{22}(0,0)=A_{23}(0,0)=A_{33}(0,0)=0.
		\end{cases}
	\end{array}
\end{equation}

We further assume that
\begin{equation}
	\label{GSUH6}
	\tag{H6}
	\begin{array}{l}
		J^3f_0(0)
		=
		bx_1x_2
		+
		x_1x_3
		\left(
		c+\displaystyle\sum_{i=1}^{3}x_i a_{13}^{i}
		\right)
		+
		x_1^2
		\left(
		1+\displaystyle\sum_{i=1}^{3}x_i a_{11}^{i}
		\right),
		\\[0.2cm]
		b=2A_{12}(0,0),
		\qquad
		c=2A_{13}(0,0),
		\qquad
		|b|=1,
		\qquad
		c\neq 0.
	\end{array}
\end{equation}

Our aim is to describe the bifurcation set $\Psi$ of the germ
$(X,(0,0))$. Here $\Psi$ denotes the smallest closed subset of
$\mathbb{R}^3$ such that the topological type of $X_\lambda$ is locally
constant on each connected component of $\mathbb{R}^3\setminus \Psi$,
with respect to topological equivalence.
\section{Main Results}
\label{Sect4_GPU}
We first identify a principal part of the unfolding $X$. This principal
part gives the leading-order geometry of the local bifurcation set.

\begin{lemma}
	\label{GSUlem1}
	Let $\Phi:\mathbb{R}^3\times \mathbb{R}\times \mathbb{R}^3
	\longrightarrow
	\mathbb{R}^3\times \mathbb{R}^3$, 	be the map $(\mu,s,y)\longmapsto (\lambda,x),
	$
	defined by
	$\lambda_1=s^6\mu_1$,	$\lambda_2=s^2\mu_2$,$\lambda_3=s\mu_3$, 
	$x_1=s^3y_1$, $x_2=s^4y_2$, and $x_3=s^5y_3.$
	Let $X\in \mathfrak{X}^{m}(\mathbb{R}^{3}\mid \mathbb{R}^{3})$
	be an unfolding of $X_0$ satisfying \eqref{GSUH4}--\eqref{GSUH6}.
	Then the following assertions hold.
	
	\begin{enumerate}
		\item For every $s\in \mathbb{R}^{\ast}$, the map $\Phi_s:=\Phi(\cdot,s,\cdot)
		$
		is an analytic diffeomorphism preserving the origin of
		$\mathbb{R}^3\times \mathbb{R}^3$. Moreover, for every
		$(\mu,s)\in \mathbb{R}^3\times \mathbb{R}^{\ast}$, the map
		$\Phi_{\mu,s}:=\Phi(\mu,s,\cdot)$
		preserves the orientation of $\mathbb{R}^3$.
		
		\item There exists a unique $C^{m-1}$-unfolding $Y$ of $X_0$, depending on
		the parameters $(\mu,s)\in \mathbb{R}^3\times \mathbb{R}$, such that, for
		every $s\neq 0$,
		$$
		(\Phi_s)_*Y_s=\frac{1}{s}X.
		$$
		Equivalently,
		$$
		(\Phi_s)_*(sY_s)=X.
		$$
		More precisely,
		$$
		Y_{\mu,s}
		=
		y_2\frac{\partial}{\partial y_1}
		+
		y_3\frac{\partial}{\partial y_2}
		+
		g(\mu,s,y)\frac{\partial}{\partial y_3},
		$$
		where $g$ is of class $C^{m-1}$ and satisfies
		\begin{equation}
			\label{GSUg}
			g(\mu,s,y)
			=
			\frac{1}{s^6}f_{\lambda}(x)\circ \Phi.
		\end{equation}
		Thus
		\begin{align*}
			g(\mu,s,y)
			&=
			\mu_1+\mu_2y_2+\mu_3y_3+y_1^2
			+
			y_1^2 A_{11}\circ \Phi
			+
			2s y_1y_2 A_{12}\circ \Phi
			\\
			&\quad
			+
			2s^2 y_1y_3 A_{13}\circ \Phi
			+
			s^2 y_2^2 A_{22}\circ \Phi
			+
			2s^3 y_2y_3 A_{23}\circ \Phi
			+
			s^4 y_3^2 A_{33}\circ \Phi.
		\end{align*}
		
		Moreover,
		\begin{enumerate}
			\item the unfolding $Y$ may be written as
			$$
			Y
			=
			Y_0
			+
			s
			\left(
			\sum_{i,j=1}^{3}y_iy_jB_{ij}
			\right)
			\frac{\partial}{\partial y_3},
			$$
			where the functions $B_{ij}$ are of class $C^{m-1}$ and are given by
			\begin{equation}\label{Aij_Bij_GPU}
				\begin{array}{ccc}
A_{11}\circ \Phi = sB_{11},&B_{12} = A_{12}\circ \Phi,& B_{13} = sA_{13}\circ \Phi, \\
				B_{22} = sA_{22}\circ \Phi,&
				B_{23} = s^2A_{23}\circ \Phi,&
				B_{33} = s^3A_{33}\circ \Phi.
			\end{array}
			\end{equation}
			
			\item For every $(\mu,s)\in \mathbb{R}^3\times \mathbb{R}^{\ast}$,
			one has $\operatorname{sgn}(\lambda_1)=\operatorname{sgn}(\mu_1),$
			because $\lambda_1=s^6\mu_1.$
		\end{enumerate}
	\end{enumerate}
\end{lemma}

\begin{proof}
	\begin{enumerate}
		\item For fixed $s\neq 0$, the map $\Phi_s$ is linear and invertible.
		Its Jacobian determinant on the full space
		$\mathbb{R}^3\times \mathbb{R}^3$ is $s^6\cdot s^2\cdot s\cdot s^3\cdot s^4\cdot s^5=s^{21}\neq 0.$
		Hence $\Phi_s$ is an analytic diffeomorphism preserving the origin.		
		For fixed $(\mu,s)$, the induced map on the phase variables is $y\longmapsto x=(s^3y_1,s^4y_2,s^5y_3),$
		whose Jacobian determinant is $s^3s^4s^5=s^{12}>0.$
		Therefore $\Phi_{\mu,s}$ preserves the orientation of $\mathbb{R}^3$.
		
		\item The vector field $X$ is defined by $\dot{x}_1=x_2$,
		$\dot{x}_2=x_3$, $\dot{x}_3=f_\lambda(x)$, $\dot{\lambda}_i=0.$
		We look for $Y$ in the form
		$$
		Y_{\mu,s}
		=
		y_2\frac{\partial}{\partial y_1}
		+
		y_3\frac{\partial}{\partial y_2}
		+
		g(\mu,s,y)\frac{\partial}{\partial y_3}.
		$$
		Because $x_1=s^3y_1$, $x_2=s^4y_2$, $x_3=s^5y_3$,
		we obtain
		$$
		(\Phi_s)_*Y_s
		=
		s^3y_2\frac{\partial}{\partial x_1}
		+
		s^4y_3\frac{\partial}{\partial x_2}
		+
		s^5g(\mu,s,y)\frac{\partial}{\partial x_3}.
		$$
		On the other hand,
		$$
		\frac{1}{s}X
		=
		s^3y_2\frac{\partial}{\partial x_1}
		+
		s^4y_3\frac{\partial}{\partial x_2}
		+
		\frac{1}{s}f_\lambda(x)\frac{\partial}{\partial x_3}.
		$$
		Thus the equality
		$$
		(\Phi_s)_*Y_s=\frac{1}{s}X
		$$
		is equivalent to
		$$
		s^5g(\mu,s,y)=\frac{1}{s}f_\lambda(x),
		$$
		that is,
		$$
		g(\mu,s,y)=\frac{1}{s^6}f_\lambda(x)\circ\Phi.
		$$
		
		Substituting the expressions of $\lambda_i$ and $x_i$ into
		$f_\lambda(x)$ gives
		\begin{align*}
			g(\mu,s,y)
			&=
			\mu_1+\mu_2y_2+\mu_3y_3+y_1^2
			+
			y_1^2A_{11}\circ\Phi
			+
			2sy_1y_2A_{12}\circ\Phi
			\\
			&\quad
			+
			2s^2y_1y_3A_{13}\circ\Phi
			+
			s^2y_2^2A_{22}\circ\Phi
			+
			2s^3y_2y_3A_{23}\circ\Phi
			+
			s^4y_3^2A_{33}\circ\Phi.
		\end{align*}
		
		Since $A_{11}(0,0)=0$, the composition $A_{11}\circ\Phi$ is divisible
		by $s$. Hence there exists a $C^{m-1}$ function $B_{11}$ such that
		$A_{11}\circ\Phi=sB_{11}.$
	
	By setting \eqref{Aij_Bij_GPU}, we obtain
		$$
		Y
		=
		Y_0
		+
		s
		\left(
		\sum_{i,j=1}^{3}y_iy_jB_{ij}
		\right)
		\frac{\partial}{\partial y_3}.
		$$
		
		Finally, since $\lambda_1=s^6\mu_1$ and $s^6>0$ for $s\neq 0$, we have
		$\operatorname{sgn}(\lambda_1)=\operatorname{sgn}(\mu_1).$
	\end{enumerate}
\end{proof}

\begin{definition}
	The vector field $$Y_0=y_2\frac{\partial}{\partial y_1}+y_3\frac{\partial}{\partial y_2}+
	\left(
	u_1+u_2y_2+u_3y_3+y_1^2
	\right)
	\frac{\partial}{\partial y_3}$$
	is called the principal part of $X$.
\end{definition}

To study the structural stability of local bifurcations of the germ of an
unfolding at a critical point, it is necessary to establish the persistence
of the topological type under sufficiently small perturbations. This is
usually achieved by restricting the vector field to a local center manifold
and by approximating the reduced dynamics up to an order that determines the
topological type.

By Lemma \ref{GSUlem1}, for $s\neq 0$, the unfolding $X$ is orbitally
equivalent to $Y_s$. Moreover, $Y_s$ is a perturbation of the principal part
$Y_0$. Therefore, the structural stability of a local bifurcation of $Y_0$
implies the structural stability of the corresponding local bifurcation of
$X$.

We now identify the local bifurcation values of the germ of the unfolding
$X$ at the origin of $\mathbb{R}^3\times \mathbb{R}^3$.

\begin{theorem}
	Let $X\in \mathfrak{X}^{m}(\mathbb{R}^{3}\mid \mathbb{R}^{3})$
	satisfy hypotheses \eqref{GSUH1}--\eqref{GSUH3}. Let
	$J(\lambda,x)$ denote the Jacobian matrix of $X_\lambda$ with respect to
	$x$. If $r_1\in \mathbb{R}$, $r_2=\alpha+i\beta$, $r_3=\alpha-i\beta$
	are the eigenvalues of $J(\lambda,x)$, then the germ
	$(X_\lambda,0)$ possesses the following local bifurcations.
	
	\begin{enumerate}
		\item Three singularities of codimension one.
		
		\begin{enumerate}
			\item The singularity $	x^1=(0,0,0)$
			appears on the surface
			$$
			S_1=
			\left\{
			\lambda\in \mathbb{R}^3:
			\lambda_1=0,\;
			\lambda_3\neq 0,\;
			4\lambda_2+\lambda_3^2\leq 0
			\right\}.
			$$
			The eigenvalues are $r_1=0$,	and
			$$
			r_{2,3}
			=
			\frac{1}{2}
			\left(
			\lambda_3
			\pm
			i\left|4\lambda_2+\lambda_3^2\right|^{1/2}
			\right).
			$$
			
			\item For $\lambda_1<0$, let $v_\varepsilon	=-\varepsilon\sqrt{-\lambda_1}$, $\varepsilon=\pm 1,$
			be the corresponding solution, up to higher-order terms, of
			$$
			\lambda_1+x_1^2
			\left(
			1+A_{11}(\lambda,x_1)
			\right)=0.
			$$
			Then $x^\varepsilon=(v_\varepsilon,0,0)$, $\varepsilon=\pm 1,$
			are singularities associated with the surface
			$$
			S_2^\varepsilon
			=
			\left\{
			\lambda:
			\gamma_{2\varepsilon}<0,\;
			\gamma_{3\varepsilon}\neq 0,\;
			\gamma_{1\varepsilon}
			+
			\gamma_{2\varepsilon}\gamma_{3\varepsilon}=0
			\right\},
			$$
			where
			$$
			\gamma_{i\varepsilon}
			=
			\frac{\partial f_\lambda}{\partial x_i}
			(\lambda,x^\varepsilon),
			\qquad
			i=1,2,3.
			$$
			The eigenvalues are $r_1=\gamma_{3\varepsilon},$ $r_2=-r_3=i\sqrt{-\gamma_{2\varepsilon}}.$
		\end{enumerate}
		
		\item Two singularities of codimension two.
		
		\begin{enumerate}
			\item For $\lambda\in C_1=\{0\}\times \mathbb{R}_{-}^{\ast}\times \{0\},$
			the singularity $x^2=(0,0,0)$
			has eigenvalues $r_1=0$, $r_2=-r_3=i\sqrt{-\lambda_2}.$
			
			\item For $\lambda\in C_2=\{0\}\times \{0\}\times \mathbb{R}^{\ast},$
			the singularity $x^3=(0,0,0)$ has eigenvalues $r_1=\lambda_3,$ 
			$r_2=r_3=0.$
		\end{enumerate}
	\end{enumerate}
\end{theorem}

\begin{proof}
	The singularities of $X_\lambda$ are the solutions of
	$$
	x_2=0,
	\qquad
	x_3=0,
	\qquad
	f_\lambda(x_1,0,0)=0.
	$$
	Using \eqref{GSUH5}, this last equation becomes
	$$
	\lambda_1
	+
	x_1^2
	\left(
	1+A_{11}(\lambda,x_1,0,0)
	\right)
	=0.
	$$
	
	At a singularity $x^\ast=(v,0,0)$, the Jacobian matrix of $X_\lambda$ is
	$$
	J(\lambda,x^\ast)
	=
	\begin{pmatrix}
		0 & 1 & 0\\
		0 & 0 & 1\\
		\gamma_1 & \gamma_2 & \gamma_3
	\end{pmatrix},
	$$
	where
	$$
	\gamma_i
	=
	\frac{\partial f_\lambda}{\partial x_i}
	(\lambda,x^\ast),
	\qquad
	i=1,2,3.
	$$
	The characteristic polynomial is therefore
	$$
	P(r)
	=
	r^3-\gamma_3r^2-\gamma_2r-\gamma_1.
	$$
	
	At the origin $x^1=(0,0,0)$, when $\lambda_1=0$, one has
	$$
	\gamma_1=0,
	\qquad
	\gamma_2=\lambda_2,
	\qquad
	\gamma_3=\lambda_3.
	$$
	Thus
	$$
	P(r)
	=
	r(r^2-\lambda_3r-\lambda_2).
	$$
	Hence one eigenvalue is $r_1=0,$
	and the other two are
	$$
	r_{2,3}
	=
	\frac{1}{2}
	\left(
	\lambda_3
	\pm
	\sqrt{\lambda_3^2+4\lambda_2}
	\right).
	$$
	If $\lambda_3^2+4\lambda_2\leq 0,$
	then this pair is complex conjugate and may be written as
	$$
	r_{2,3}
	=
	\frac{1}{2}
	\left(
	\lambda_3
	\pm
	i\left|4\lambda_2+\lambda_3^2\right|^{1/2}
	\right).
	$$
	This gives the first codimension-one case.
	
	Now assume $\lambda_1<0$. Then the equation
	$$
	\lambda_1
	+
	x_1^2
	\left(
	1+A_{11}(\lambda,x_1,0,0)
	\right)
	=0
	$$
	has two nearby solutions
	$$
	v_\varepsilon
	=
	-\varepsilon\sqrt{-\lambda_1}
	+
	\text{higher-order terms},
	\qquad
	\varepsilon=\pm 1.
	$$
	Thus there are two singularities $x^\varepsilon=(v_\varepsilon,0,0).$
	At such a point, the eigenvalues are the roots of
	$$
	r^3-\gamma_{3\varepsilon}r^2
	-\gamma_{2\varepsilon}r
	-\gamma_{1\varepsilon}
	=0.
	$$
	If this polynomial has one real eigenvalue $r_1=\gamma_{3\varepsilon}$
	and a purely imaginary pair $r_{2,3}=\pm i\sqrt{-\gamma_{2\varepsilon}},$
	then necessarily $\gamma_{2\varepsilon}<0$
	and $\gamma_{1\varepsilon}+\gamma_{2\varepsilon}\gamma_{3\varepsilon}=0.$
	This gives the surfaces $S_2^\varepsilon$.
	
	Finally, the codimension-two cases occur when degeneracies in the previous
	codimension-one surfaces overlap.
	
	If $\lambda_1=0,$ $\lambda_3=0,$ $\lambda_2<0,$
	then at the origin the characteristic polynomial is $
	P(r)=r(r^2-\lambda_2),$
	and therefore $r_1=0,$ $r_{2,3}=\pm i\sqrt{-\lambda_2}.$
	This gives the curve $C_1$.
	
	If $\lambda_1=0,$ $\lambda_2=0,$ $\lambda_3\neq 0,$
	then $P(r)=r^2(r-\lambda_3),$
	so that $r_1=\lambda_3,$ $r_2=r_3=0.$
	This gives the curve $C_2$.
\end{proof}

\begin{theorem}\label{Theo2_GPU}
	Let $X\in \mathfrak{X}^{m}(\mathbb{R}^{3}\mid \mathbb{R}^{3})$
	satisfy hypotheses \eqref{GSUH1}--\eqref{GSUH3}. For
	$\varepsilon=\pm 1$, let $x^\varepsilon(\lambda)$ denote the two
	singularities of $X_\lambda$ near the origin, and set
	$$
	\gamma_{i\varepsilon}(\lambda)
	=
	\frac{\partial f_\lambda}{\partial x_i}
	\bigl(x^\varepsilon(\lambda)\bigr),
	\qquad i=1,2,3.
	$$
	Define the germs of curves $\Gamma_\varepsilon$ in the parameter space by
	$$
	\Gamma_\varepsilon
	=
	\left\{
	\lambda\in \mathbb{R}^3:
	\exists \alpha\in \mathbb{R},\;
	\gamma_{1\varepsilon}(\lambda)=\varepsilon \alpha^3,\;
	\gamma_{2\varepsilon}(\lambda)=\alpha^2,\;
	\gamma_{3\varepsilon}(\lambda)=\alpha
	\right\}.
	$$
	Then there exists $\alpha_0>0$ such that, for every
	$0<|\alpha|<\alpha_0$ and every $\lambda\in \Gamma_\varepsilon$
	corresponding to $\alpha$, the vector field $X_\lambda$ has exactly two
	singularities near the origin, denoted by
	$x^+(\lambda)$, $x^-(\lambda),$
	and these singularities are hyperbolic saddles.
	
	Moreover, if $\alpha>0$, then
	$$
	W^u(X_\lambda,x^-)
	\cap
	W^s(X_\lambda,x^+)
	\neq \varnothing.
	$$
	If $\alpha<0$, then
	$$
	W^u(X_\lambda,x^+)
	\cap
	W^s(X_\lambda,x^-)
	\neq \varnothing.
	$$
	In both cases, the intersection is transverse.
\end{theorem}
Theorem~\ref{Theo2_GPU} comes from thefollowing result.
\begin{lemma}
	\label{GSUlemHetPrincipal}
	There exists $\alpha_0>0$ such that, for every
	$0<|\alpha|<\alpha_0$, the vector field
	$Y_{0,\mu(\alpha)}$
	with
	$$
	\mu(\alpha)
	=
	\left(
	-\frac14\alpha^6,\alpha^2,\alpha
	\right)
	$$
	has two hyperbolic singularities
	$$
	y^\pm
	=
	\left(
	\pm \frac12\alpha^3,0,0
	\right).
	$$
	If $\alpha>0$, then
	$$
	W^u(Y_{0,\mu(\alpha)},y^-)
	\cap
	W^s(Y_{0,\mu(\alpha)},y^+)
	\neq \varnothing.
	$$
	If $\alpha<0$, then
	$$
	W^u(Y_{0,\mu(\alpha)},y^+)
	\cap
	W^s(Y_{0,\mu(\alpha)},y^-)
	\neq \varnothing.
	$$
	Moreover, the corresponding intersection is transverse.
\end{lemma}

\begin{proof}
	We prove the result by reducing the dynamics of $X_\lambda$ to the
	principal part $Y_0$.
	
	First, by Lemma \ref{GSUlem1}, after the weighted change of variables
	$$
	\lambda_1=s^6\mu_1,
	\qquad
	\lambda_2=s^2\mu_2,
	\qquad
	\lambda_3=s\mu_3,
	$$
	and
	$$
	x_1=s^3y_1,
	\qquad
	x_2=s^4y_2,
	\qquad
	x_3=s^5y_3,
	$$
	the unfolding $X$ is orbitally equivalent to an unfolding $Y_s$ of the
	principal part $Y_0$. More precisely, for $s\neq 0$,
	$$
	(\Phi_s)_*Y_s=\frac{1}{s}X.
	$$
	Equivalently,
	$$
	(\Phi_s)_*(sY_s)=X.
	$$
	Thus, up to a smooth change of coordinates and a time reparametrization,
	the local phase portrait of $X_\lambda$ is the same as that of $Y_s$.
	
	Moreover, $Y_s=Y_0+	sR_s,$
	where $R_s$ is of class $C^{m-1}$. Hence $Y_s$ is a $C^1$-small
	perturbation of $Y_0$ when $s$ is sufficiently small.
	
	Now consider
	$$
	\mu(\alpha)
	=
	\left(
	-\frac14\alpha^6,\alpha^2,\alpha
	\right).
	$$
	For this parameter value, the singularities of $Y_{0,\mu(\alpha)}$ are
	given by
	$$
	y^\pm
	=
	\left(
	\pm\frac12\alpha^3,0,0
	\right),
	$$
	because they solve
	$$
	y_2=0,\qquad y_3=0,\qquad
	-\frac14\alpha^6+y_1^2=0.
	$$
	
	At such a singularity, the Jacobian matrix of $Y_{0,\mu(\alpha)}$ is
	$$
	J_\pm
	=
	\begin{pmatrix}
		0 & 1 & 0\\
		0 & 0 & 1\\
		\pm \alpha^3 & \alpha^2 & \alpha
	\end{pmatrix}.
	$$
	Its characteristic polynomial is
	$$
	P_\pm(r)
	=
	r^3-\alpha r^2-\alpha^2r\mp \alpha^3.
	$$
	Putting $r=\alpha q$, we get
	$$
	P_\pm(\alpha q)
	=
	\alpha^3
	\left(
	q^3-q^2-q\mp 1
	\right).
	$$
	For $\alpha\neq 0$, none of these polynomials has roots on the imaginary
	axis. Therefore $y^+$ and $y^-$ are hyperbolic singularities.
	
	By Lemma \ref{GSUlemHetPrincipal}, for $0<|\alpha|<\alpha_0$ the vector
	field $Y_{0,\mu(\alpha)}$ has a transverse heteroclinic orbit joining
	$y^-$ to $y^+$ if $\alpha>0$, and joining $y^+$ to $y^-$ if $\alpha<0$.
	
	Since the intersection of the corresponding invariant manifolds is
	transverse, it persists under sufficiently small $C^1$ perturbations.
	Consequently, for $s$ sufficiently small, the perturbed vector field
	$Y_{s,\mu(\alpha)}$ also possesses a transverse heteroclinic orbit between
	the continuations of these two hyperbolic singularities.
	
	Finally, applying the diffeomorphism $\Phi_s$, the transverse
	heteroclinic orbit of $Y_s$ is sent to a transverse heteroclinic orbit of
	$X_\lambda$, with $\lambda=\Phi_s(\mu(\alpha)).$
	Since diffeomorphisms preserve transversality of intersections of
	submanifolds, we obtain
	$$
	W^u(X_\lambda,x^-)
	\cap
	W^s(X_\lambda,x^+)
	\neq \varnothing
	$$
	for $\alpha>0$, and
	$$
	W^u(X_\lambda,x^+)
	\cap
	W^s(X_\lambda,x^-)
	\neq \varnothing
	$$
	for $\alpha<0$.
	
	The transversality of these intersections follows from the transversality
	for the principal part and from the invariance of transverse intersections
	under diffeomorphisms.
	
	Therefore the germs of curves $\Gamma_\pm$ are contained in the local
	bifurcation set of the unfolding $X$.
\end{proof}

\section{Conclusion and Research Perspectives}

In this work, we constructed a generic three-parameter unfolding of a germ
of three-dimensional vector fields exhibiting a completely degenerate
singularity at the origin. The local analysis carried out in this paper
allowed us to describe the associated bifurcation diagram and to prove the
existence of a one-parameter family of heteroclinic orbits connecting
distinct equilibrium points. These results reveal the richness and
complexity of the local dynamics that may arise near such highly
degenerate singularities.

Nevertheless, several fundamental questions remain open. In particular,
the possible existence of homoclinic orbits has not yet been fully
investigated. Such orbits could lead to more intricate dynamical behavior,
including the emergence of chaotic dynamics. Moreover, a complete
classification of the generic bifurcations occurring in this setting is
still lacking. A refined stratification of the parameter space according
to the different bifurcation types would provide a deeper and more global
understanding of the transitions between the corresponding dynamical
regimes.

Future research may therefore focus on the following directions:
\begin{itemize}
	\item studying the robustness of the bifurcation structures obtained
	under higher-order perturbations;
	\item extending the present results to complex, time-dependent, or
	non-autonomous frameworks;
	\item carrying out a global analysis of invariant manifolds and their
	mutual interactions in a larger neighbourhood of the degenerate
	singularity;
	\item investigating the possible occurrence of homoclinic connections
	and their implications for chaotic dynamics;
	\item developing a complete stratification of the parameter space in
	terms of local and global bifurcation phenomena.
\end{itemize}

These perspectives indicate that completely degenerate singular germs
constitute a rich source of problems at the intersection of bifurcation
theory, qualitative dynamics, and the geometry of dynamical systems.


\begin{thebibliography}{99}

\bibitem{FR} J. R. Francoise, and  R. Roussarie, Bifurcations of Planar Vector Fields, Springer verlag, Ed.1, 1480 (1991).
\bibitem{GP}J. Guckenheimer and P. Holmes, Nonlinear Oscillations, Dynamical Systems, and Bifurcations of Vector Fields, Springer Science, Ed.1, 42 (1983).
\bibitem{PL} L. Perko, Differential Equations and Dynamical Systems, Springer Science, Ed.3, 7(2013). 
\bibitem{i}
A. Sayizonga, and A. Musesa, Sur les modèles Topologiques de germes singuliers de champs de vecteurs sur $\mathbb{R}^{3}$ de jet linéaire nilpotent, Annales de la Faculté des Sciences de l'UNIKIN, 2(2013), 17-20.
\bibitem{iv} A. Sayizonga, Déploiement générique d'un germe singulier de champs de vecteurs sur $\mathbb{R}^{3}$ de codimension Trois, Thèse, University of Kinshasa, 2012.
\bibitem{v} F. Takens, Singularities of Vector Fields; Publ. Math. IHES, 43(1974), 47-100.
\bibitem{TNB4} B. Tangue Ndawa.
Dynamics on Bi-Lagrangian Structures and Cherry maps.  	arXiv:2508.12350.
\end{thebibliography}
\end{document}